\documentclass[10pt]{amsart}

\usepackage{amssymb,amsmath,amscd,amsthm,amsfonts,wasysym,mathrsfs,amsxtra,multirow}
\usepackage{verbatim}
\input xy
\xyoption{all}

\newcommand{\OO}{\mathscr{O}}

\newcommand{\MM}{\mathscr{M}}

\newcommand{\PP}{\mathbb{P}}

\newcommand{\kernel}{\textnormal{ker}\,}
\newcommand{\cokernel}{\textnormal{coker}\,}
\newcommand{\degree}{\textnormal{deg }}

\newcommand{\HHom}{\mathscr{H}om} 
\newcommand{\Hom}{\textnormal{Hom}}
\newcommand{\dimension}{\textnormal{dim}\,}

\newcommand{\rank}{\textnormal{rk}\,}

\newcommand{\Ext}{\textnormal{Ext}}
\newcommand{\EExt}{\mathscr{E}xt}

\newcommand{\Aut}{\textnormal{Aut}}

\newcommand{\al}{\alpha}
\newcommand{\cone}{\textnormal{cone}}

\newcommand{\Ac}{\mathcal{A}}

\newcommand{\Coh}{\textnormal{Coh}}

\newcommand{\Ap}{\mathcal{A}^p}

\newcommand{\arinj}{\ar@{^{(}->}}
\newcommand{\arsurj}{\ar@{->>}}
\newcommand{\areq}{\ar@{=}}

\renewcommand{\geq}{\geqslant}
\renewcommand{\leq}{\leqslant}
\renewcommand{\tilde}{\widetilde}
\renewcommand{\bar}{\overline}

\newcommand{\ra}{\rightarrow}
\newcommand{\lra}{\longrightarrow}

\def\b{\mathbf{b}}                           
\renewcommand{\P}{\mathbb{P}}

\newcommand{\cO}{\mathcal{O}}

\newtheorem{theorem}{Theorem}[section]
\newtheorem{lemma}[theorem]{Lemma}
\newtheorem{coro}[theorem]{Corollary}
\newtheorem{pro}[theorem]{Proposition}

\theoremstyle{definition}

\theoremstyle{remark}
\newtheorem{remark}[theorem]{Remark}

\begin{document}

\title{Representing stable complexes on projective spaces}

\author{Jason Lo}
\address{Department of Mathematics \\ 202 Mathematical Sciences Building \\ University of Missouri\\  Columbia MO 65211 \\ USA}
\curraddr{Taimali, Taiwan}
\email{jccl@alumni.stanford.edu}

\author[Ziyu Zhang]{Ziyu Zhang}
\address{Max-Planck Institute for Mathematics\\ Vivatsgasse 7, 53111 Bonn\\ Germany}
\curraddr{Department of Mathematics\\ University of Bath\\ Claverton Down\\ Bath, BA2 7AY\\ United Kingdom}
\email{zz505@bath.ac.uk}

\thanks{}


\begin{abstract}
We give an explicit proof of a Bogomolov-type inequality for $c_3$ of reflexive sheaves on $\PP^3$.  Then, using resolutions of rank-two reflexive sheaves on $\PP^3$, we prove that the closed points of some strata of the moduli of rank-two complexes that are both PT-stable and dual-PT-stable can be given the structure of quotient stacks.  Using monads, we apply the same techniques to $\PP^2$ and obtain similar results for  some strata of the moduli of Bridgeland-stable complexes.
\end{abstract}

\maketitle

\section{Introduction}

Let $X$ be any smooth projective threefold over $k$.  In previous work \cite[Section 4.2]{Lo3}, we considered the moduli functor
\begin{equation}\label{eq1}
\coprod_n \MM^{PT\cap PT^\ast}_{(r,d,\beta,n)}
\end{equation}
where the points $[E]$ of each  moduli functor $\MM^{PT\cap PT^\ast}_{(r,d,\beta,n)}$ represent complexes $E \in D^b(X)$  satisfying:
\begin{itemize}
\item[] $E$ is both PT-stable and PT-dual stable, and $ch(E)= (r,d,\beta,n)$.
\end{itemize}
We also  observed in \cite[Section 4.2]{Lo3} that, when $r$ and $d$ are both integers that are coprime,  the points of \eqref{eq1} are in 1-1 correspondence with pairs of the form $([F],[Q^D])$, where
\begin{itemize}
\item $[F]$ is the isomorphism class of a $\mu$-stable reflexive sheaf $F$ on $X$ with
 \[
 (ch_0(F[1]), ch_1(F[1]), ch_2(F[1]))=(r,d,\beta);
 \]
 \item $[Q^D]$ is the isomorphism class of the dual $Q^D := R\HHom (Q,\OO_X)[3]$ of $Q$, where  $Q$ is a quotient of the 0-dimensional sheaf $\EExt^1 (F,\OO_X)$.
\end{itemize}
Under this correspondence, we have $F = H^{-1}(E)$ and $Q^D = H^0(E)$ for any point $[E]$ of $\coprod_n \MM^{PT\cap PT^\ast}_{(r,d,\beta,n)}$.   We can therefore think of the moduli \eqref{eq1} as parametrising $\mu$-stable reflexive sheaves $F$, with each isomorphism class occurring with multiplicity equal to the number of non-isomorphic quotients of $\EExt^1 (F,\OO_X)$.

This paper grew out of an attempt to find a more concrete and down-to-earth description of the objects parametrised by the moduli stack \eqref{eq1}, with the hope that it might help us understand whether the stack \eqref{eq1} is a quotient stack.  Objects in the derived category are often considered difficult to work with, because of the presence of quasi-isomorphisms (i.e.\ two very different-looking complexes can be isomorphic in the derived category).  In Sections \ref{sec-quotstackP3} and \ref{sec-P2}, we show that isomorphisms in the derived category for the objects in \eqref{eq1} can be understood as isomorphisms between diagrams in the category of coherent sheaves.

In Section \ref{sec-Bog}, we give a Bogomolov-type inequality for $\mu$-semistable reflexive sheaves $F$ on $\mathbb{P}^3$.  We point out, that it is already known that there is a bound for $ch_3$ in terms of $ch_0, ch_1$ and $ch_2$ for $\mu$-semistable reflexive sheaves on a smooth variety over a field of characteristic zero.  This is implicit in the proof of \cite[Main Theorem]{Mar}, for example (see also \cite[Section 3]{Langer} and \cite{Tod}).  However, we write down such an explicit bound for $ch_3$ in Theorem \ref{_theorem_main_}.  The proof of Theorem \ref{_theorem_main_} follows closely the ideas in \cite{BR}, and is deferred to Section \ref{section-proofmaintheorem}.  As an immediate consequence of this theorem, we obtain that the moduli stack \eqref{eq1} is of finite type (Corollary \ref{theorem-bigfunctorfinite}).

In Section \ref{sec-quotstackP3}, we build on the work of Mir\'{o}-Roig \cite{MR} and use particular 2-term locally free resolutions of reflexive sheaves to prove Theorem \ref{_quotient_coro_}, which says that the closed points of certain strata of the moduli stack \eqref{eq1} are in bijection with the closed points of certain quotient stacks, when $X = \mathbb{P}^3, ch_0=-2$ and $ch_1=1$.

In Section \ref{sec-P2}, we adapt the techniques in Section \ref{sec-quotstackP3} from $\PP^3$ to $\PP^2$, using the results on monads due to  Jardim \cite{Jardim-inst} and Jardim-Martins \cite{Jardim-Martins}.  This culminates in Theorem \ref{P2main-result},  that the closed points of certain strata of the moduli of Bridgeland-semistable objects in $D(\PP^2)$ are in bijection with the closed points of some quotient stacks.  We  hope that this brings us  one step closer to  understanding the connections between Bridgeland stability and the moduli of monads, a question posed at the end of the article \cite{JMR} by Jardim-Mir\'{o}-Roig.

\bigskip
\noindent
\textbf{Notation.} All schemes will be over an algebraically closed  field $k$ of characteristic zero. For a variety $X$, we will write $\Coh (X)$ to denote the category of coherent sheaves on $X$, and $D(X)$ to denote the bounded derived category of coherent sheaves.  For a coherent sheaf $F$ on $X$, we write $F^\ast$ to denote the usual sheaf dual, i.e.\ $F^\ast := \HHom (F,\OO_X)$.  For objects $E \in D(X)$, we write $E^\vee$ to denote the derived dual, i.e.\ $E^\vee := R\HHom (E,\OO_X)$; we also write $H^i(E)$ to denote the degree-$i$ cohomology of $E$.

For integers $i<j$, we write $D^{[i,j]}(X)=D^{[i,j]}_{\Coh (X)}(X)$ to denote the category of objects $E$ in $D(X)$ such that $H^s(E)=0$ for all $s<i$ and $s>j$.  For any nonnegative integer $d$, we write $\Coh_{\leq d}(X)$ to denote the category of coherent sheaves on $X$ supported in dimension at most $d$, and $\Coh_{\geq d}$ to denote the category of coherent sheaves on $X$ with no subsheaves supported in dimension lower than $d$.  For integers $0 \leq d < e$, we write $\langle \Coh_{\leq d}(X), \Coh_{\geq e}(X)[1]\rangle$ to denote the extension-closed subcategory of $D(X)$ generated by $\Coh_{\leq d}(X)$ and $\Coh_{\geq e}(X)[1]$; that is, the objects in $\langle \Coh_{\leq d}(X), \Coh_{\geq e}(X)[1]\rangle$ are the complexes $E$ such that $H^{-1}(E) \in \Coh_{\geq e}(X)$ and $H^0(E) \in \Coh_{\leq d}(X)$, and $H^s(E)=0$ for all $s \neq -1, 0$.

On a smooth projective threefold $X$, we write $\mathbb{D}(-)$ to denote the dualising functor $-^\vee[2] =R\HHom (-,\OO_X)[2]$.

\bigskip
\noindent
\textbf{Acknowledgements.} The authors would like to thank Zhenbo Qin for helpful discussions, Marcos Jardim for answering our questions on monads, and Arend Bayer for many helpful comments and pointing out an error in an earlier version of the manuscript.  The second author would also like to thank Zhenbo Qin and Dan Edidin, for the invitation to visit the University of Missouri, where part of this work was completed. He would also like to thank the support of SFB/TR 45 and Max-Planck-Institute for Mathematics. The authors are also grateful to the referee, whose helpful comments have led to quite a few improvements.

\section{A Bogomolov-type inequality for $\mu$-semistable reflexive sheaves}\label{sec-Bog}

We have the following Bogomolov-type inequality for $ch_3$ of $\mu$-semistable reflexive sheaves on $\mathbb{P}^3$:

\begin{theorem}\label{_theorem_main_}
Let $F$ be a $\mu$-semistable reflexive sheaf on $\mathbb{P}^3$.  Writing $r = ch_0 (F)$, $c_1 = ch_1(F)$, $ch_2 = ch_2(F)$ and $ch_3 = ch_3(F)$, we have the following bound of $ch_3(F)$ that only depends on $ch_0(F), ch_1(F)$ and $ch_2(F)$:
\begin{eqnarray}
|ch_3| &<& 2\left(\frac{|c_1|}{r}+r+4-ch_2+\frac12\sum_{i=1}^r \left(\frac{|c_1|}{r}+r\right)^2\right) \notag\\
& & \cdot \left(-ch_2+\frac12\sum_{j=1}^r \left(\frac{|c_1|}{r}+r\right)^2\right)+\frac{n}{6} \left( \frac{|c_1|}{r}+r+3 \right)^3 \notag\\
& & +\left(2|ch_2|+\frac{11}{6}|c_1|+r\right), \label{eq18}
\end{eqnarray}
\end{theorem}

\begin{coro}\label{theorem-bigfunctorfinite}
When $X = \mathbb{P}^3$, the moduli space \eqref{eq1}  is of finite type over $k$.
\end{coro}
\begin{proof}
This follows immediately from \cite[Proposition 3.4]{Lo1} and Theorem \ref{_theorem_main_}.
\end{proof}

Since the proof of Theorem \ref{_theorem_main_} is a little long, we have placed it in Section  \ref{section-proofmaintheorem}.

\section{Quotient stacks of stable complexes}\label{sec-quotstackP3}

The aim of this section is to show that the closed points of certain strata of the moduli stack \eqref{eq1} can be given the structure of quotient stacks when $X = \mathbb{P}^3, r=-2$ and $c_1=1$ .

\subsection{When $X$ is an arbitrary smooth projective threefold}

Let us start with the following observation:  for any smooth projective variety $X$ of dimension $n$ over a field $k$, and any complex $E \in D(X)$ with cohomology only in degrees $-1$ and $0$ such that $H^0(E)$ is supported in dimension 0, write $F = H^{-1}(E)$ and $Q = H^0(E)$.  Then
$E$ is represented by a class in $\Ext^2_{D(X)}(Q,F)$, where
\begin{align}
  \Ext^2_{D(X)} (Q,F) &\cong \Hom (Q,F[2])  \notag\\
   &\cong \Hom_{D(X)} (F^\vee [-2],Q^\vee) \notag\\
   &\cong \Hom_{D(X)} (F^\vee [-2],Q^D) \text{ where $Q^D := \EExt^{n} (Q,\OO_X)$} \notag\\
   &\cong \Hom_{D(X)} (F^\vee [n-2],Q^D) \label{eqn-transform}
\end{align}
The conditions that $E \in D^{[-1,0]}_{\Coh (X)}(X)$ and $H^0(E)$ is 0-dimensional are satisfied by, for instance, complexes that are polynomial stable at the large volume limit when $X$ is a  surface \cite[Lemma 4.2]{Bayer}, and complexes that are stable with respect to PT-stability when $X$ is a threefold \cite[Lemma 3.3]{Lo1}.

Now, let $X$ be a smooth projective threefold.  Write $\Ap$ to denote the subcategory $\langle \Coh_{\leq 1}(X), \Coh_{\geq 2}(X)[1]\rangle$ of $D(X)$.  Note that $\Ap$ is the heart of a bounded t-structure on $D(X)$, and is an Abelian category.   When $r$ and $d$  are coprime integers,  the points $[E]$ of \eqref{eq1} are in 1-1 correspondence with isomorphism classes of complexes $E$ of Chern character $ch=(r,d,\beta,n)$ satisfying (see \cite[Proposition 4.3]{Lo3}):
\begin{itemize}
\item $E \in \Ap$;
\item $H^{-1}(E)$ is a $\mu$-stable reflexive sheaf;
\item $H^0(E)$ is a 0-dimensional sheaf;
\item the map $H^2 (\delta)$ is surjective, where $\delta$ is the connecting morphism in the exact triangle in $D(X)$
    \begin{equation}\label{eqn-canonicaltridual}
      H^0(E)^\vee \to E^\vee \to H^{-1}(E)^\vee[-1] \overset{\delta}{\to} H^0(E)^\vee[1],
    \end{equation}
    which is obtained by dualising the canonical exact triangle
    \begin{equation}\label{eqn-canonicaltri}
      H^{-1}(E)[1] \to E \to H^0(E) \to H^{-1}(E)[2].
    \end{equation}
\end{itemize}

As a first step towards producing strata of \eqref{eq1} whose closed points have the structure of quotient stacks, we consider the following two sets:
\begin{align*}
A_1 = \{ E \in \Ap : \, &H^{-1}(E) \text{ is reflexive}, \\
&H^0(E) \in \Coh_{\leq 0}(X), \\
    & \text{the map }H^2(\delta) : \EExt^1(H^{-1}(E),\OO_X) \to H^0(E) \text{ is surjective} \}
\end{align*}
and
\begin{align*}
A_2 = \{ \text{morphisms } F^\vee[1] \overset{t}{\to} Q \text{ in }\Ap: \, &t \text{ is a surjection in }\Ap, \\
&F  \text{ is a reflexive sheaf}, \\
&  Q \in \Coh_{\leq 0}(X)\}.
\end{align*}
On the set $A_1$, we define the equivalence relation $\thicksim_1$ to be isomorphism in the derived category $D(X)$.  For objects in $A_2$ (which are morphisms in the derived category of the form $F^\vee [1] \to Q$), we say two  morphisms $F_1^\vee[1] \to Q_1$ and $F_2^\vee[1] \to Q_2$ in $\Ap$ are equivalent with respect to $\thicksim_2$  if there is a commutative diagram in the derived cateogory
\begin{equation}\label{eq11}
\xymatrix{
 F_1^\vee[1] \ar[r] \ar[d]^\cong & Q_1 \ar[d]^\cong \\
 F_2^\vee[1] \ar[r] & Q_2
}
\end{equation}
where the vertical maps are isomorphisms.  We have:

\begin{pro}\label{pro-objmor1to1corr}
Let $X$ be a smooth projective threefold over $k$.   There is a bijection between the set of equivalence classes $A_1/\thicksim_1$ and $A_2/\thicksim_2$.
\end{pro}

\begin{proof}
 Take an equivalence class $[E]$  in $A_1/\thicksim_1$, and write $F = H^{-1}(E)$ and $Q = H^0(E)$.  Using the canonical exact triangle \eqref{eqn-canonicaltri}, we can consider $E$ as a representative of a class in $\Ext^2(Q,F)$.  The  string of isomorphisms \eqref{eqn-transform} with $n=3$ sends the complex $E$ to a morphism $t : F^\vee[1] \to Q^D$.

Consider the composition
\[
  F^\vee[1] \overset{c}{\twoheadrightarrow} H^0(F^\vee[1]) \overset{H^2(\delta)}{\twoheadrightarrow} H^3(Q^\vee)=Q^D,
\]
where $c$ is the canonical morphism; this composite map is exactly the morphism $t$ from the previous paragraph.  The composite is also a surjection in $\Ap$,  because $c$ is a surjection in $\Ap$, while $H^2(\delta)$ is a surjection in $\Coh_{\leq 0} (X)$ (by the definition of $A_1$), hence a surjection in $\Ap$.  It is then easy to see that sending $E$ to $t$ gives us a well-defined function $f : A_1/\thicksim_1 \to A_2/\thicksim_2$.

Given any morphism $t : F^\vee[1] \to Q^D$ in $A_2$, we can obtain an object $E \in \Ap$ with $H^{-1}(E)\cong F$ and $H^0(E)\cong Q$ using the string of isomorphisms \eqref{eqn-transform}.  The surjectivity of $t$ as a morphism in $\Ap$ implies the surjectivity of $H^0(t)=H^2(\delta)$.  Hence $f$ is surjective.

Lastly, let us show the injectivity of $f$.  Suppose two equivalence classes $[E_1], [E_2]$ in $A_1/\thicksim_1$ are taken by $f$ to the same equivalence class in $A_2/\thicksim_2$.  Then $H^{-1}(E_1)^\vee [1] \cong H^{-1}(E_2)^\vee [1]$ and $H^0(E_1) \cong H^0(E_2)$.  Therefore, $H^i (E_1) \cong H^i(E_2)$ for all $i$.  Replacing the $E_i$ by isomorphic complexes in the derived category if necessary, we can assume that $H^i(E_1)= H^i(E_2)$ for all $i$.  Write $F = H^{-1}(E_i)$ and $Q' = H^3 (H^0(E_i)^\vee)$.  That $f([E_1])=f([E_2])$ means there is a commutative diagram in $\Ap$
\[
\xymatrix{
  F^\vee[1] \ar[r]^{t_1} \ar[d]_\thicksim^{j_1} & Q' \ar[d]_\thicksim^{j_2}  \\
  F^\vee[1] \ar[r]^{t_2} & Q'
}
\]
where each $t_i$ is the morphism that $f$ sends $E_i$ to, and each  $j_i$ is an isomorphism.  Then the $E_i^\vee$ can be recovered as the cones of the morphisms $t_i$  (up to shift).  Hence the $E_i$ must be isomorphic complexes in the derived category.  Therefore, $f$ is a bijection between $A_1/\thicksim_1$ and $A_2/\thicksim_2$ as claimed.
\end{proof}

\begin{remark}
 From the proof of Proposition \ref{pro-objmor1to1corr}, we can see that if $E \in A_1$ and $f$ sends the equivalence class of $E$ to that of the morphism $t : F^\vee[1] \to Q$, then $\cone (t)$ is isomorphic to $E^\vee[3]$.  Hence the map $A_2 \to A_1$ that takes a morphism $t : F^\vee[1] \to Q$ to the object $\cone (t)[-1]$ induces the inverse of $f$.
\end{remark}

Let us  define
\begin{align*}
  A_1^s(ch_0,ch_1,ch_2) := \{ E \in A_1 : \, &H^{-1}(E) \text{ is $\mu$-stable, and} \\
     & ch_i(H^{-1}(E)) = ch_i \text{ for $0\leq i \leq 2$} \}; \\
  A_2^s(ch_0,ch_1,ch_2) := \{ F^\vee[1] \overset{t}{\to} Q \text{ in } A_2 : & \, F \text{ is $\mu$-stable, and} \\
  & ch_i(F) = ch_i \text{ for $0\leq i \leq 2$} \}.
\end{align*}
Then Proposition \ref{pro-objmor1to1corr} immediately gives us a bijection
\begin{equation}\label{eq9}
  A_1^s (ch_0,ch_1,ch_2)/\thicksim_1 \leftrightarrow A_2^s (ch_0,ch_1,ch_2)/\thicksim_2
\end{equation}
for any $ch_0,ch_1,ch_2$.  On the other hand, if we write $\left| \coprod_n \MM^{PT \cap PT^\ast}_{(r,d,\beta,n)} \right|$ to denote the set of $k$-valued points of the moduli stack $\coprod_n \MM^{PT \cap PT^\ast}_{(r,d,\beta,n)}$, then \cite[Proposition 4.3]{Lo3} gives
\begin{equation}\label{eq8}
\left| \coprod_n \MM^{PT \cap PT^\ast}_{(-r,-d,-\beta,n)} \right| = A_1^s (r,d,\beta)/\thicksim_1
\end{equation}
when $r,d$ are coprime integers.

We can also define $$A_1^s(ch;l)=\{E\in A_1^s(ch_0,ch_1,ch_2): ch_3(H^{-1}(E))=ch_3, \textnormal{length}(H^0(E))=l\}.$$
Then putting \eqref{eq9} and \eqref{eq8} together, we have shown that there is a series of bijections
\begin{multline}\label{eq19}
\left| \coprod_{ch_3} \MM^{PT \cap PT^\ast}_{-ch} \right| = A_1^s (ch_0,ch_1,ch_2)/\thicksim_1
= \coprod_{ch_3,l} A_1^s(ch;l)/\thicksim_1\\
\leftrightarrow \coprod_{ch_3,l} A_2^s(ch;l)/\thicksim_2 = A_2^s (ch_0,ch_1,ch_2)/\thicksim_2
\end{multline}
where we define
\[
A_2^s(ch;l) := \{ F^\vee[1] \to Q \text{ in } A_2^s(ch_0,ch_1,ch_2) : \, ch(F)=ch, \text{length}(Q)=l\}.
\]

We further point out that, the first half of equation \eqref{eq19} can even be understood on the level
of moduli functors. In fact, if we define $\mathcal{A}_1^s(ch;l)$ to be the substack of $\coprod_{ch_3} \MM^{PT \cap PT^\ast}_{-ch}$ consisting of complexes in $A_1^s(ch;l)$, then

\begin{lemma}\label{_stratification_lemma_}
We have a locally closed stratification of the moduli functor \eqref{eq1} as
\begin{equation}\label{_stratification_functor_}
\coprod_{ch_3} \MM^{PT \cap PT^\ast}_{-ch} = \coprod_{ch_3,l}\mathcal{A}_1^s(ch;l).
\end{equation}
\end{lemma}

\begin{proof}
Since we already know that both sides parametrize exactly the same objects, we are left to show that each
functor $\mathcal{A}_1^s(ch;l)$ is a locally closed subfunctor of the functor \eqref{eq1}. In fact, whenever
we have a flat family of complexes over a scheme $B$ in the functor \eqref{eq1}, we can stratify its base $B$ into
locally closed subschemes, so that the strata are indexed by the Chern character of $H^{-1}(E)$ and the length
of $H^0(E)$, where $[E]$ is a fiber of the family. This shows that the  morphism of moduli stacks
$$\mathcal{A}_1^s(ch;l) \hookrightarrow \coprod_{ch_3} \MM^{PT \cap PT^\ast}_{-ch}$$
is a locally closed embedding.
\end{proof}

\subsection{When $X = \mathbb{P}^3$}

Reflexive sheaves on a smooth projective threefold have two-term locally free resolutions.  Stable reflexive sheaves on $\PP^3$ with particular Chern characters  have very special two-term resolutions, in which all the terms are direct sums of twists of $\OO_{\PP^3}(1)$  (e.g.\ see \cite{MR}).  We now take advantage of such resolutions.  In the rest of this section, we fix $X = \mathbb{P}_k^3$.

By \cite[Lemma 2.7]{MR}, if $F$ is a stable reflexive sheaf on $\PP^3$ with rank $r$ and Chern classes $c_i$ satisfying
\begin{align}
&r=2, \, c_1=-1, \, c_2>4,  \notag\\
& c_3 = c_2^2-2sc_2+2s(s+1) \text{ where } 1 \leq s \leq (-1+\sqrt{4c_2-7})/2, \label{eq4}
\end{align}
then $F$ has a locally free resolution of the form
\begin{equation*}
  0 \to R^{-1} \to R^0 \to F \to 0
\end{equation*}
where
\begin{align}
  R^{-1} &= \OO_{\PP^3} (-s-2) \oplus \OO_{\PP^3} (s-1-c_2),  \notag\\
  R^0 &= \OO_{\PP^3} (-s-1)\oplus \OO_{\PP^3} (-1) \oplus \OO_{\PP^3} (-2) \oplus \OO_{\PP^3} (s-c_2). \label{eq5}
\end{align}
Therefore, for any  reflexive sheaf $F$ satisfying \eqref{eq4}, if we consider $F$ as a 1-term complex sitting at degree 0 in the derived category, then the object $F^\vee[1]$ is isomorphic in the derived category to a 2-term complex of the form $[(R^0)^\ast \to (R^{-1})^\ast]$ that sits at degrees $-1$ and 0.

As in the proof of \cite[Proposition 2.9]{MR},  there is an open subspace $V' \subset V := \PP (\Hom_{\Coh (X)} (R^{-1},R^0))$ such that the closed points of $V'$, up to the actions of the automorphism groups $\Aut (R^{-1})$ and $\Aut (R^0)$,  are in bijection with all the stable reflexive sheaves $F$ satisfying \eqref{eq4}, up to isomorphism.  The bijection is given by
\[
  (R^{-1} \overset{g}{\to} R^0) \mapsto \cokernel (g).
\]
 With  $R^{-1}, R^0$ as in \eqref{eq5}, let $R^\bullet$ denote a complex of the form $[R^{-1} \overset{g}{\to} R^0]$ for some map $g$, sitting at degrees $-1$ and 0. Then $d(R^\bullet) := (R^\bullet)^\vee[1]$ is also a 2-term complex of locally free sheaves sitting at degrees $-1$ and 0.  For any  coherent  sheaf $Q$, we have $\Hom_{D(\PP^3)}(d(R^\bullet),Q) \cong \Hom_{\Coh (\PP^3)}(H^0(d(R^\bullet)),Q)$.  This means that any morphism $d(R^\bullet) \to Q$ in the derived category $D(\PP^3)$ factors through the canonical map $d(R^\bullet) \to H^0(d(R^\bullet))$.  Therefore, the set
\[
  \{ f \in \Hom_{D(\PP^3)}(d(R^\bullet),Q) : H^0(f) \text{ is surjective in } \Coh (\PP^3) \}
\]
is in bijection with
\[
  \{ f \in \Hom_{\Coh(\PP^3)}(H^0(d(R^\bullet)),Q) : f \text{ is surjective in }\Coh (\PP^3)\}.
\]

Now, with $R^{-1}, R^0$ still as in \eqref{eq5}, let $A_3^s$ denote the set of diagrams in $\Coh (\mathbb{P}^3)$ of the form
\[
  (R^0)^\ast \overset{a}{\to} (R^{-1})^\ast \overset{b}{\to} Q
\]
satisfying the following conditions:
\begin{enumerate}
\item  $a$ is the dual of a map $R^{-1} \overset{g}{\to} R^0$ that corresponds to a point in $V'$;
\item $Q$ is a 0-dimensional sheaf on $\PP^3$;
\item $b$ is a surjective map of coherent sheaves;
\item the composition $ab=0$.
\end{enumerate}
We then place an equivalence relation $\thicksim_3$ on the set $A_3^s$, where two such diagrams
\[
    (R^0)^\ast \overset{a_1}{\to} (R^{-1})^\ast \overset{b_1}{\to} Q_1
\]
and
\[
    (R^0)^\ast \overset{a_2}{\to} (R^{-1})^\ast \overset{b_2}{\to} Q_1
\]
are defined to be equivalent with respect to $\thicksim_3$ if there is a commutative diagram in $\Coh (\PP^3)$
\[
\xymatrix{
  (R^0)^\ast \ar[r]^{a_1} \ar[d] & (R^{-1})^\ast \ar[r]^{b_1} \ar[d] & Q_1 \ar[d] \\
   (R^0)^\ast \ar[r]^{a_2}  & (R^{-1})^\ast \ar[r]^{b_2} & Q_2
}
\]
where all vertical arrows are isomorphisms of sheaves.  For any Chern character $ch=(ch_0,ch_1,ch_2,ch_3)$ on $X$ and any nonnegative integer $l$, let us also define
\begin{equation*}
 A_3^s(ch;l) = \{  (R^0)^\ast \overset{a}{\to} (R^{-1})^\ast \overset{b}{\to} Q  \text{ in } A_3^s : ch(\cokernel(a^\ast))=ch, \text{length}(Q)=l \}.
\end{equation*}
Then we have
\[
  A_2^s (ch_0,ch_1,ch_2) = \coprod_{ch_3,l} A_2^s (ch;l)
\]
and the following lemma:

\begin{lemma}\label{lemma6}
Let $X = \PP^3$.  For any $ch$ satisfying \eqref{eq4} and any $l$, there is a bijection
 \begin{equation}\label{eq10}
 A_2^s (ch;l)/\thicksim_2 \leftrightarrow A_3^s(ch;l)/\thicksim_3.
 \end{equation}
\end{lemma}


\begin{proof}
We can define a map $h: A_3^s(ch;l) \to A_2^s(ch;l)$ by taking a diagram $(R^0)^\ast \overset{a}{\to} (R^{-1})^\ast \overset{b}{\to} Q$ to the following composition of morphisms in $D(\PP^3)$
\begin{equation}\label{eq16}
\xymatrix{
  & F^\vee[1] \ar[d]^\cong \\
 [(R^0)^\ast \ar[r]^a & (R^{-1})^\ast \ar[d]^b] \\
  & Q
}
\end{equation}
where we consider $[(R^0)^\ast \overset{a}{\to} (R^{-1})^\ast]$ as a complex sitting at degrees $-1$ and 0, and where the upper vertical map is any isomorphism in $D(\PP^3)$, $F$ being any  rank-two stable reflexive sheaf with resolution $0 \to R^{-1} \overset{a^\ast}{\to} R^0 \to F \to 0$.  It is clear that this operation induces a well-defined map $\tilde{h} : A_3^s(ch;l)/\thicksim_2 \to A_2^s (ch;l)/\thicksim_3$.

We begin by checking that $\tilde{h}$ is surjective.  Take any $F^\vee[1] \overset{t}{\to} Q$  in $A_2^s(ch;l)$.  Then $F$ has a resolution of the form $R^{-1} \overset{a^\ast}{\to} R^0$ for some $a$.  Write $d(R^\bullet)$ to denote the complex $[(R^0)^\ast \overset{a}{\to} (R^{-1})^\ast]$ (whose terms are at degrees $-1$ and $0$), and fix an  isomorphism $d(R^\bullet)  \overset{u}{\to} F^\vee[1]$ in $D(\PP^3)$.  Since $\Hom_{D(\PP^3)}(d(R^\bullet),Q) \cong \Hom_{\Coh (\PP^3)}(H^0(d(R^\bullet)),Q)$, the composition $tu$ factorises as
\[
d(R^\bullet) \to H^0(d(R^\bullet)) \to Q
\]
where the first map is the canonical map, a genuine chain map, and the second map is a map of coherent sheaves.  Therefore, $tu$ can be represented by a chain map from $d(R^\bullet)$ to $Q$.  We thus get a diagram of the form
\[
  (R^0)^\ast \to (R^{-1})^\ast \to Q,
\]
which is an object in $A_3^s(ch;l)$, and it is taken by $h$ into the equivalence class of $t$. Hence $\tilde{h}$ is surjective.

Next, we show that $\tilde{h}$ is injective.  Suppose we have two diagrams of the form
\begin{equation}\label{eqn-hinjtwodiagrams}
    (R^0)^\ast \overset{a_i}{\to} (R^{-1})^\ast \overset{b_i}{\to} Q_i
\end{equation}
in $A_3^s(ch;l)$ for $i=1,2$.  Let $F_i$ be the cokernel of $a_i^\ast : R^{-1} \to R^0$ for each $i$, and write $e_i$ to denote the chain map $[R^{-1} \overset{a_i^\ast}{\to} R^0] \to F_i$ induced by the short exact sequence
\[
 0 \to R^{-1} \overset{a_i^\ast}{\to} R^0 \to F_i \to 0 \text{ for each $i$}.
\]
Write $d(R^\bullet_i)$ to denote the complex $[(R^0)^\ast \overset{a_i}{\to} (R^{-1})^\ast]$ sitting at degrees $-1$ and 0 for each $i$.  Then $\tilde{h}$ takes the diagrams \eqref{eqn-hinjtwodiagrams} to the $\thicksim_2$-equivalence classes of the composite morphisms
\[
\xymatrix{
  F_i^\vee [1] \overset{e_i^\vee [1]}{\to} d(R_i^\bullet) \overset{b_i}{\to} Q_i.
  }
\]
(Here, we are abusing notation and writing $b_i$ to also denote the chain map it induces.)  Let $t_i=b_i e_i^\vee[1]$ for $i=1,2$.  Suppose that $t_1$ and $t_2$ are equivalent with respect to $\thicksim_2$.  Then by the definition of $\thicksim_2$ and our construction of $\tilde{h}$, there is a commutative diagram in $D(\PP^3)$
\begin{equation}\label{eqn-A2A3bigcommute}
\xymatrix{
  F_1^\vee[1] \ar[r]^{e_1^\vee[1]}_\cong \ar[d]^\cong_\beta & d(R^\bullet_1) \ar[r]^{b_1} &  Q_1 \ar[d]^\cong_\gamma \\
  F_2^\vee[1] \ar[r]^{e_2^\vee[1]}_\cong & d(R^\bullet_2) \ar[r]^{b_2}  & Q_2
}.
\end{equation}

    From the proof of \cite[Proposition 2.9]{MR}, the points of $V'$ correspond to isomorphism classes of rank-two stable reflexive sheaves on $\PP^3$.  Since $F_1$ and $F_2$ are isomorphic via the sheaf isomorphism   $\beta^\vee [1] : F_2 \to F_1$, there is a chain map
\begin{equation}\label{eq6}
\xymatrix{
  R^{-1} \ar[r]^{a_1^\ast} \ar[d]^{r^{-1}} & R^0 \ar[r] \ar[d]^{r^0} & F_1 \ar[d]_\cong^{(\beta^\vee[1])^{-1}} \\
  R^{-1} \ar[r]^{a_2^\ast} & R^0 \ar[r] & F_2
  }
\end{equation}
where $r^{-1}, r^0$ are also sheaf isomorphisms.  Taking the derived dual of \eqref{eq6} and shifting by 1, we obtain a commutative diagram in the derived category $D(\mathbb{P}^3)$
\begin{equation*}
\xymatrix{
  F_1^\vee [1] \ar[d]^\beta \ar[r] & d(R^\bullet_1) \ar[d]^{(r^\bullet)^\vee [1]} \\
  F_2^\vee [1] \ar[r] & d(R^\bullet_2)
  }
\end{equation*}
in which all arrows are isomorphisms in the derived category, and the horizontal maps are merely $e_i^\vee[1]$.  Taking the inverses of the horizontal maps, we obtain a commutative diagram
\begin{equation}\label{eq7}
\xymatrix{
d(R^\bullet_1) \ar[d]_{(r^\bullet)^\vee [1]}  \ar[r]^{(e_1^\vee[1])^{-1}} &  F_1^\vee [1] \ar[d]^\beta \\
d(R^\bullet_2)  \ar[r]^{(e_2^\vee[1])^{-1}} &  F_2^\vee [1]
  }
\end{equation}
in which all arrows are still isomorphisms in the derived category.  Concatenating  diagram \eqref{eqn-A2A3bigcommute} with  diagram \eqref{eq7}, we obtain  a commutative square
\begin{equation}
\xymatrix{
  d(R^\bullet_1) \ar[r]^{b_1} \ar[d]_{(r^\bullet)^\vee [1]} & Q_1 \ar[d]^\gamma_\cong \\
  d(R^\bullet_2) \ar[r]^{b_2} & Q_2
}
\end{equation}
in the derived category $D(\mathbb{P}^3)$.

The chain map ${(r^\bullet)^\vee [1]}$ induces the morphism $H^0(d(R_1^\bullet)) \to H^0(d(R_2^\bullet))$ between cohomology sheaves at degree 0.  Also, for each $i$, we have $\Hom_{D(\PP^3)}(d(R_i^\bullet),Q_i) \cong \Hom_{\Coh (\PP^3)} (H^0(d(R_i^\bullet)),Q_i)$, which means that $b_i$ factors in the derived category as
\[
d(R_i^\bullet) \overset{\phi_i}{\to} H^0(d(R_i^\bullet)) \overset{\psi_i}{\to} Q_i,
\]
where $\phi_i$ is the canonical map (which is a chain map) and $\psi_i = H^0(b_i)$ is a map in $\Coh (\PP^3)$.  We thus obtain a diagram
\begin{equation}
\xymatrix{
  d(R_1^\bullet) \ar[d]^{(r^\bullet)^\vee[1]} \ar[r]^{\phi_1} & H^0(d(R_1^\bullet)) \ar[r]^{\psi_1} \ar[d]^{H^0((r^\bullet)^\vee[1])} & Q_1 \ar[d]^\gamma \\
  d(R_2^\bullet) \ar[r]^{\phi_2} & H^0(d(R_2^\bullet)) \ar[r]^{\psi_2} & Q_2
}
\end{equation}
which is commutative in the category of chain complexes of coherent sheaves on $\PP^3$.  The outer edges of this commutative diagram then induces the following commutative diagram in  $\Coh (\PP^3)$
\[
\xymatrix{
  (R^0)^\ast \ar[r]^{a_1} \ar[d]^{(r^0)^\ast} & (R^{-1})^\ast \ar[r]^{\theta_1} \ar[d]^{(r^{-1})^\ast} & Q_1 \ar[d]^\gamma \\
  (R^0)^\ast \ar[r]^{a_2} & (R^{-1})^\ast \ar[r]^{\theta_2} & Q_2
}.
\]
Recall that the map $\theta_i$ is the same map as $b_i$ in \eqref{eqn-hinjtwodiagrams}.  Hence the two diagrams \eqref{eqn-hinjtwodiagrams} for $i=1,2$ are equivalent with respect to $\thicksim_3$, i.e.\ $\tilde{h}$ is injective.
\end{proof}

The following lemma will be the final step in showing that the closed points of certain strata of the moduli stack \eqref{eq1} have the structure of quotient stacks:
\begin{lemma}\label{_quotient_lemma_}
Let $X = \PP^3$.  For any $ch$ satisfying \eqref{eq4} and any $l$, the set
$$A_3^s(ch;l)/\thicksim_3$$
is in bijection with the set of closed points of a quotient stack.
\end{lemma}


\begin{proof}
By our definition of $A_3^s(ch;l)$, the maps $$a: (R^0)^\ast \ra (R^{-1})^\ast$$ that appear in elements in $A_3^s (ch;l)$ are parametrised by a quasi-projective scheme $V'$.  Since $(R^{-1})^\ast$ is a fixed sheaf as constructed in \eqref{eq5}, and any $Q$ that appears in an element of $A_3^s(ch;l)$ is a $0$-dimensional
quotient sheaf of $(R^{-1})^\ast$ with fixed length $l$, and hence of  fixed Hilbert polynomial,
 to each map $$b: (R^{-1})^\ast \ra Q$$ that appears in an element of $A_3^s(ch;l)$ we can associate a point of the Grothendieck quot scheme $Quot((R^{-1})^\ast,l).$  In other words, we have a set-theoretic map
 \[
   \phi : A_3^s(ch;l) \to V' \times Quot((R^{-1})^\ast,l)
 \]
 that sends the diagram $(R^0)^\ast \overset{a}{\to} (R^{-1})^\ast \overset{b}{\to} Q$ to $(a,b)$. Since  we require $ba=0$ for elements in $A_3^s(ch;l)$, the image of $\phi$ is a closed subscheme $W$ of $V' \times Quot((R^{-1})^\ast,l)$, which is a quasi-projective scheme.  The equivalence relation $\thicksim_3$ on $A_3^s(ch;l)$ induces an equivalence relation on $W$, which we also denote by $\thicksim_3$.

We claim that, two elements in $W$ are equivalent with respect to $\thicksim_3$ if and only if they differ by an action
of the group
\begin{eqnarray*}
G&=&\Aut(R^{-1})\times\Aut(R^0)\\
&=&\Aut((R^{-1})^\ast)\times\Aut((R^0)^\ast).
\end{eqnarray*}

First of all, we consider the action of a group element $$(\alpha,\beta)\in\Aut((R^{-1})^\ast)\times\Aut((R^0)^\ast)$$
on an element $$(R^0)^\ast \overset{a}{\to} (R^{-1})^\ast \overset{b}{\to} Q$$ of $A_3^s(ch;l)$.   It produces again
an element in $A_3^s(ch;l)$, as in the second row of the following diagram
\begin{equation*}
\xymatrix{
(R^0)^\ast   \ar[r]^a   \ar[d]^\alpha   &   (R^{-1})^\ast   \ar[r]^b   \ar[d]^\beta   &   Q   \ar[d]^=   \\
(R^0)^\ast   \ar[r]^{a'}   &   (R^{-1})^\ast   \ar[r]^{b'}   &   Q
}
\end{equation*}
where
\begin{eqnarray*}
a'&=&\beta  a \alpha^{-1};\\
b'&=&b  \beta^{-1}.
\end{eqnarray*}
It's clear that both squares commute, and all vertical arrows are isomorphisms, which are required in the
equivalence relation $\thicksim_3$. Therefore the group $G$ does act on $A_3^s(ch;l)$.  It is easy to see that $G$ induces a well-defined action on $W$ that makes $\phi : A_3^s(ch;l) \twoheadrightarrow W$ a $G$-equivariant map.

Now we want to show that if two elements in $A_3^s(ch;l)$ are equivalent under $\thicksim_3$, then up to an action of  the group $G$, they map under $\phi$ to the same element in $W$. Assume we have the following commutative diagram in $\Coh (\PP^3)$
\begin{equation*}
\xymatrix{
(R^0)^\ast   \ar[r]^{a_1}   \ar[d]^\alpha   &   (R^{-1})^\ast   \ar[r]^{b_1}   \ar[d]^\beta   &   Q   \ar[d]^\gamma   \\
(R^0)^\ast   \ar[r]^{a_2}   &   (R^{-1})^\ast   \ar[r]^{b_2}   &   Q
}
\end{equation*}
where all the vertical arrows are isomorphisms. We immediately have
\begin{eqnarray*}
\alpha &\in& \Aut((R^{0})^\ast);\\
\beta &\in& \Aut((R^{-1})^\ast).
\end{eqnarray*}

Note that we can factorise the above chain map in the following manner:
\begin{equation*}
\xymatrix{
(R^0)^\ast   \ar[r]^{a_1}   \ar[d]^=   &   (R^{-1})^\ast   \ar[r]^{b_1}   \ar[d]^=   &   Q   \ar[d]^\gamma   \\
(R^0)^\ast   \ar[r]^{a_1}   \ar[d]^\alpha   &   (R^{-1})^\ast   \ar[r]^{b_2 \circ \beta}   \ar[d]^\beta   &   Q   \ar[d]^=   \\
(R^0)^\ast   \ar[r]^{a_2}   &   (R^{-1})^\ast   \ar[r]^{b_2}   &   Q
}
\end{equation*}
In the above diagram, the first and the second rows are mapped by $\phi$ to  the same element in $W\subseteq (V'\times Quot((R^{-1})^\ast,l))$; this is because points $Q$ of the quot scheme $Quot((R^{-1})^\ast,l)$ are considered the same if they differ by an automorphism as quotients of $(R^{-1})^\ast$.   The second and the third rows obviously differ by an action of $(\alpha, \beta)\in G$.

Therefore, we have shown that two elements in $W$ are equivalent under
the relation $\thicksim_3$ if and only if one can be obtained from the other via the action of a group element in $G$.  It follows that $A_3^s(ch;l)/\thicksim_3$ is in bijection with the closed points of  a global quotient stack, namely
$$A_3^s(ch;l)/\thicksim_3= W/\thicksim_3=[W/G].$$
\end{proof}

Finally, we are ready to use Lemma \ref{_quotient_lemma_} to describe some of the strata in Lemma
\ref{_stratification_lemma_}.

\begin{theorem}\label{_quotient_coro_}
Let $X=\PP^3$. Then for any stratum of \eqref{_stratification_functor_} with $ch$ satisfying \eqref{eq4}, the set
of underlying closed points is in bijection with the closed points of a global quotient stack.
\end{theorem}

\begin{proof}
The set of underlying closed points in the stratum $\mathcal{A}_1^s(ch;l)$ is given by the set $A_1^s(ch;l)/\thicksim_1$.
When $ch$ satisfies \eqref{eq4}, we know from \eqref{eq19} and \eqref{eq10} that it is in bijection with $A_3^s(ch;l)/\thicksim_3$.
The result then follows from Lemma \ref{_quotient_lemma_}.
\end{proof}

\begin{remark}\label{remark2}
Theorem \ref{_quotient_coro_} identifies stable complexes with closed points of a quotient stack. It is a natural question to ask if the stratum of the moduli stack under consideration is isomorphic to the above quotient stack $[W/G]$. Unfortunately, it is \emph{not} the case. A more detailed analysis can be found in Appendix \ref{app}.
\end{remark}

\section{Monads and stable complexes on surfaces}\label{sec-P2}

 Recall that a monad on a variety $X$ is defined as a three-term complex of locally free sheaves
\begin{equation}\label{eqn-monaddef}
0 \to R^{-1} \to R^0 \to R^1 \to 0
\end{equation}
with $R^i$ at degree $i$, and which is exact everywhere except perhaps at degree 0.  The cohomology of the monad is defined as the cohomology of the complex at degree 0.  We have:

\begin{lemma}\label{lemma-FveeQchainmaprep}
Let $X$ be a smooth projective variety.  Let $Q$ be a 0-dimensional sheaf on $X$, and $F$ a coherent sheaf on $X$ that is the cohomology of a monad $0 \to R^{-1} \to R^0 \to R^1 \to 0$.  Then any morphism $t : F^\vee \to Q$ in the derived category $D(X)$ can be represented by a chain map of the form
\begin{equation}
  \xymatrix{
    (R^1)^\ast \ar[r]^\al & (R^0)^\ast \ar[r] \ar[d] & (R^{-1})^\ast  \\
     & Q. &
  }
\end{equation}
In other words, the vertical map factors through the canonical map $(R^0)^\ast \twoheadrightarrow \cokernel (\al)$.
\end{lemma}
\begin{proof}
In the derived category $D(X)$, if we consider $F$ as a one-term complex with $F$ at degree 0, then $F$ is  isomorphic to the complex $[R^{-1} \to R^0 \to R^1]$ with $R^0$ at degree 0.  So $F^\vee$ is isomorphic to $[(R^1)^\ast \overset{\al}{\to} (R^0)^\ast \overset{\beta}{\to} (R^{-1})^\ast]$ in the derived category.

Let $A^\bullet$ denote the 2-term complex $[(R^1)^\ast \overset{\al}{\to} (R^0)^\ast]$ with $(R^0)^\ast$ sitting at degree 1, and $B^\bullet$ be the 1-term complex with $(R^{-1})^\ast$ sitting at degree 1.  Then we have an exact triangle in $D(X)$
\[
  A^\bullet \to B^\bullet \to F^\vee \to A^\bullet [1].
\]
Applying $\Hom_{D(X)} (-,Q)$ to this exact triangle, we get an exact sequence
\[
 \Hom (A^\bullet [1], Q)\to \Hom (F^\vee,Q) \to \Hom (B^\bullet, Q),
\]
where $\Hom (B^\bullet, Q)=\Hom ((R^{-1})^\ast[-1], Q)=\Ext^1 ((R^{-1})^\ast,Q)\cong H^1 (X, (R^{-1}) \otimes Q)$ (recall that $R^{-1}$ is locally free), which is zero  since $Q$ is 0-dimensional.  Hence any morphism $t : F^\vee \to Q$ in the derived category is induced by a map $\bar{t} : A^\bullet [1] \to Q$ in the derived category.  However, $\Hom_{D(X)} (A^\bullet [1], Q) \cong \Hom_{\Coh (X)} (H^1 (A^\bullet),Q) = \Hom_{\Coh (X)} (\cokernel (\al),Q)$, i.e.\ $\bar{t}$  in turn factors through the canonical map $(R^0)^\ast \twoheadrightarrow \cokernel (\al)$.  The claim then follows.
\end{proof}

Following the terminology in \cite{Jardim-inst}, by a linear monad we  mean a monad on $\PP^n$  of the form
\[
  M^\bullet := [0 \to \OO_{\PP^n}(-1)^{\oplus a} \overset{\al}{\to} \OO_{\PP^n}^{\oplus b} \overset{\beta}{\to} \OO_{\PP^n}(1)^{\oplus c} \to 0].
\]
We say a coherent sheaf on $\PP^n$ is a linear sheaf if it is the cohomology of a linear monad.  Also, recall that a coherent sheaf $F$ is said to be normalised if $$-\rank (F) +1 \leq c_1 (F) \leq 0.$$  Note that for a normalised sheaf, we have $-1 < \mu (F)  \leq 0$.

Jardim has the following criterion for a sheaf on $\mathbb{P}^n$ to be a linear sheaf:

\begin{theorem}\cite[Theorem 3]{Jardim-inst}\label{thm-Jardim-inst-thm3}
If $F$ is a torsion-free sheaf on $\PP^n$ satisfying:
\begin{enumerate}
\item[(i)] for $n \geq 2$, $H^0(F(-1))=H^n(F(-n))=0$;
\item[(ii)] for $n \geq 3$, $H^1 (F(-2)) = H^{n-1} (F(1-n))=0$;
\item[(iii)] for $n \geq 4, 2 \leq p \leq n-2$ and all $k$, $H^p (F(k))=0$;
\end{enumerate}
then $F$ is a linear sheaf, and can be represented as the cohomology of the monad
\begin{equation}\label{eqn-monad1}
  0 \to H^1 (F \otimes \Omega_{\PP^n}^2 (1)) \otimes \OO_{\PP^n}(-1) \to H^1 (F \otimes \Omega_{\PP^n}^1) \otimes \OO_{\PP^n} \to H^1 (F (-1)) \otimes \OO_{\PP^n}(1) \to 0.
\end{equation}
\end{theorem}

We have the following easy observation:

\begin{lemma}\label{lemma-stnorislin}
  Let $F$ be a normalised $\mu$-semistable torsion-free sheaf on $\PP^n$.  Then $H^0(F(-1))=H^n(F(-n))=0$ (i.e.\ condition (i) of Theorem \ref{thm-Jardim-inst-thm3} is satisfied).
\end{lemma}
\begin{proof}
Note that $H^0(F(-1)) \cong \Hom (\OO_{\PP^n},F(-1))$.  Since $\OO_{\PP^n}, F(-1)$ are both $\mu$-semistable and $\mu (\OO_{\PP^n})=0 > \mu (F(-1))$, we have $\Hom (\OO_{\PP^n},F(-1)) = H^0(F(-1))=0$.  On the other hand,
\begin{align*}
  H^n (F(-n)) &\cong \Ext^n (\OO_{\PP^n}, F(-n)) \\
  &\cong \Ext^0 (F(-n), \omega_{\PP^n})  \text{ by Serre duality}\\
  &\cong \Hom (F,\OO (-1)) \text{ since $\omega_{\PP^n} \cong \OO (-n-1)$}.
  \end{align*}
Now, since $F, \OO (-1)$ are both $\mu$-semistable and $\mu (F) > -1 = \mu (\OO (-1))$, we obtain the vanishing of $\Hom (F,\OO (-1))$, and hence of $H^n (F(-n))$.
\end{proof}
On $\mathbb{P}^2$, we can be a bit more specific:
\begin{lemma}\label{lemma-norstonP2hasmonad}
Let $F$ be a normalised $\mu$-semistable torsion-free sheaf on $\PP^2$.  Then $F$ is linear and is the cohomology of a monad of the form
\begin{equation}\label{eqn-P2norstmonad}
0\to \OO_{\PP^2}(-1)^{\oplus d+c} \to \OO_{\PP^2}^{\oplus r+d+2c} \to \OO_{\PP^2}(1)^{\oplus c} \to 0
\end{equation}
where $c = -\chi (F(-1)), r=\rank (F)$ and $d=\degree (F)$.
\end{lemma}
\begin{proof}
By Lemma \ref{lemma-stnorislin}, we have $H^0(F(-1))=0$.  From the short exact sequence $0 \to F(k-1) \to F(k) \to F(k) |_H \to 0$ where $H \subset \PP^2$ is any hyperplane (i.e.\ a line in our case of $\PP^2$), and that $H^2(F(-2))=0$ (also by Lemma \ref{lemma-stnorislin}), we obtain $H^2 (F(-1))=0$.  Hence $\dimension H^1 (F(-1)) = - \chi (F(-1))$, which we denote by $c$.

In the monad \eqref{eqn-monad1} for $F$, let $v = \dimension H^1 (F \otimes \Omega_{\PP^2}^2 (1))$, $w=\dimension H^1 (F \otimes \Omega_{\PP^2})$, and $u = \dimension H^1(F(-1))$.  Then
\begin{align*}
  d &:= \degree (F) \\
  &= -\degree (H^1 (F \otimes \Omega_{\PP^n}^2 (1)) \otimes \OO_{\PP^n}(-1) ) - \degree (H^1 (F (-1)) \otimes \OO_{\PP^n}(1)) \\
  &= v-u.
\end{align*}
Similarly, $r:= \rank (F) = w-v-u$.  Hence $u=c, v=d+c$ and $w=r+v+u=r+d+2c$, giving us the lemma.
\end{proof}

Following the definition in \cite{Jardim-inst},  we say that a torsion-free sheaf $F$ on $\PP^n$ (where $n \geq 2$) is an instanton sheaf if $c_1(F)=0$ and conditions (i) through (iii) in Theorem \ref{thm-Jardim-inst-thm3} are satisfied.  We call $-\chi (F(-1))$ the charge of $F$.

\begin{remark}
In the case of $c_1(F)=0$, Lemma \ref{lemma-norstonP2hasmonad} becomes \cite[Theorem 17]{Jardim-inst}, which says that any $\mu$-semistable torsion-free sheaf on $\mathbb{P}^2$ with $c_1=0$ is instanton.
\end{remark}

Let us consider complexes $E \in D(\PP^2)$ of the following form:
\begin{itemize}
\item $H^{-1}(E)$ is a $\mu$-semistable torsion-free sheaf;
\item $H^0(E)$ is a 0-dimensional sheaf;
\item $H^i(E)=0$ for $i \neq -1, 0$.
\end{itemize}
Examples of such complexes include stable pairs (i.e.\ sections  $\OO_{\PP^2} \overset{s}{\to} F$ of sheaves $F$ where $\cokernel (s)$ is 0-dimensional) and polynomial-stable or Bridgeland-stable complexes.  Any such complex is determined by its cohomology sheaves $H^{-1}(E), H^0(E)$ and a class $\al \in \Ext^2 (H^0(E),H^{-1}(E))$.  By the isomorphisms in \eqref{eqn-transform}, we know $\al$ corresponds to a morphism $t : H^{-1}(E)^\vee \to \EExt^2 (H^0(E),\OO_{\PP^2})$ in $D(\PP^2)$.  This prompts us to define the set
\begin{align*}
  B_1^s (r,d,c;l) = \{ \text{morphisms } & F^\vee \overset{t}{\to} Q \text{ in } D(\PP^2) :  \\
    & F \text{ is a }\mu\text{-semistable}, \text{ normalised torsion-free sheaf with} \\
    &\rank (F)=r, \degree (F)=d, -\chi (F(-1))=c, \text{ and}\\
    &Q \text{ is a 0-dimensional sheaf of length }l \}.
\end{align*}
On the set $B_1^s (r,d,c;l)$, we can define an equivalence relation $\thicksim_{B_1}$ in the same manner as for $\thicksim_2$ on threefolds (see diagram \eqref{eq11}).  On the other hand, let us also define $B_2(r,d,c;l)$ to be the set of all diagrams in $\Coh (\PP^2)$ of the form
\begin{equation}\label{eq14}
  \xymatrix{
    \OO_{\PP^2}(-1)^{\oplus c} \ar[r]^\al & \OO_{\PP^2}^{\oplus r+d+2c} \ar[r]^\beta \ar[d]^\gamma & \OO_{\PP^2}(1)^{\oplus d+c} \\
    & Q &
  }
\end{equation}
such that $\al$ is injective, $\beta \al=0$ (so that the three terms in a row form a complex) and $\gamma \al=0$.

On the set $B_2(r,d,c;l)$, we define an equivalence relation $\thicksim_{B_2}$ where two such diagrams are declared equivalent if there is a commutative diagram in $\Coh (\PP^2)$ of the form
\[
\xymatrix{
\OO_{\PP^2}(-1)^{\oplus c} \ar[r]^\al \ar@/^/[drr] & \OO_{\PP^2}^{\oplus r+d+2c} \ar[d]^(.65)\gamma \ar[r]^\beta \ar@/^/[drr] & \OO_{\PP^2}(1)^{\oplus d+c} \ar@/^/[drr] & & \\
& Q \ar@/^/[drr] & \OO_{\PP^2}(-1)^{\oplus c} \ar[r]^{\al'}  & \OO_{\PP^2}^{\oplus r+d+2c} \ar[d]^(.65){\gamma'} \ar[r]^{\beta'}  & \OO_{\PP^2}(1)^{\oplus d+c} \\
& & & Q &
}
\]
in which all diagonal arrows are isomorphisms.

%

For fixed $r,d,c,l$, we can define a map
\begin{equation}\label{eq13}
  f : B_1^s (r,d,c;l) \to B_2 (r,d,c;l)
\end{equation}
in the following obvious manner: given any morphism $t : F^\vee \to Q$ in $B_1^s(r,d,c;l)$, Lemma \ref{lemma-norstonP2hasmonad} says we can find a monad of the form
\begin{equation}\label{eq12}
0\to \OO_{\PP^2}(-1)^{\oplus d+c} \to \OO_{\PP^2}^{\oplus r+d+2c} \to \OO_{\PP^2}(1)^{\oplus c} \to 0
\end{equation}
that is isomorphic to $F$ in the derived category $D(\PP^2)$.  Then $F^\vee$ is isomorphic, in the derived category, to the dual complex of \eqref{eq12}, namely
\begin{equation}\label{eq13}
\OO_{\PP^2}(-1)^{\oplus c} \overset{\al}{\to} \OO_{\PP^2}^{\oplus r+d+2c} \overset{\beta}{\to} \OO_{\PP^2}(1)^{\oplus d+c}.
\end{equation}
Note that, since $\beta$ may not be surjective, \eqref{eq13} may not be a monad.


Now, with the same argument as in the proof of Lemma \ref{lemma-FveeQchainmaprep}, we can represent the morphism $t$  by the chain map
\begin{equation*}
  \xymatrix{
    \OO_{\PP^2}(-1)^{\oplus c} \ar[r]^\al & \OO_{\PP^2}^{\oplus r+d+2c} \ar[r]^\beta \ar[d]^\gamma & \OO_{\PP^2}(1)^{\oplus d+c} \\
    & Q &
  }
\end{equation*}
where $\gamma$ is the composition of the canonical map  $\OO_{\PP^2}^{\oplus r+d+2c} \twoheadrightarrow \cokernel (\al)$ followed by some $\gamma_2$.  We have thus constructed a map $f$ from $B_1^s (r,d,c;l)$ to $B_2 (r,d,c;l)$.

Note that, in the construction of $f$ above, we had to choose a lifting $\gamma$ of $t$, and the choice is not necessarily unique.  This makes it difficult to see whether the map $f$ induces a map on the sets of equivalences classes $B_1^s(r,d,c;l)/\thicksim_{B_1}$ and $B_2 (r,d,c;l)/\thicksim_{B_2}$.  We can sidestep this problem if we add the assumption $c+d=0$:

\begin{lemma}\label{lemma7}
When $c+d=0$, there is a 1-1 correspondence between:
\begin{itemize}
\item[(a)] the isomorphism classes of normalised $\mu$-semistable torsion-free sheaf $F$ of rank $r$, degree $d$ and $-\chi (F(-1))=c$, and
\item[(b)] the closed points of an open subscheme  $V'$ of  $V$, where $V$ is the scheme parametrising isomorphism classes of  cokernels of injective morphisms  \[
    \OO_{\PP^2}(-1)^{\oplus c} \to \OO_{\PP^2}^{\oplus r+c},
    \]
    up to the actions of $\Aut (\OO_{\PP^2}(-1)^{\oplus c})$ and $\Aut (\OO_{\PP^2}^{\oplus r+c})$.
\end{itemize}
\end{lemma}

\begin{proof}
When $c+d=0$, Lemma \ref{lemma-norstonP2hasmonad} says that every sheaf $F$ of the above form is locally free,  and is the kernel of a surjection $\OO_{\PP^2}^{\oplus r+c} \to \OO_{\PP^2}(1)^{\oplus c}$.  Hence the (non-derived) dual $F^\ast$ of $F$ has a 2-term locally free resolution of the form
\[
  0 \to \OO_{\PP^2}(-1)^{\oplus c} \to \OO_{\PP^2}^{\oplus r+c} \to F^\ast \to 0.
\]
The association $F \mapsto (\OO_{\PP^2}(-1)^{\oplus c} \to \OO_{\PP^2}^{\oplus r+c})$ gives a surjective set-theoretic map from the set in (a) to the set of closed points of an  open subscheme $V'$ of $V$.  That this map is an injection is clear.
\end{proof}

Let us write
\[
 B_2^s (r,d,c;l) := \{ \text{diagrams of the form \eqref{eq14} that are in $B_2(r,d,c;l)$} : \cokernel(\al) \in V' \}.
\]

Now, similar to the case of $\PP^3$, we have:

\begin{lemma}\label{lemma8}
For fixed $r,d,c,l$ satisfying $c+d=0$, we have a bijection
\[
B_1^s(r,d,c;l)/\thicksim_{B_1} \leftrightarrow B_2^s (r,d,c;l)/\thicksim_{B_2}.
\]
\end{lemma}

\begin{proof}
Using Lemma \ref{lemma7}, the proof of Lemma \ref{lemma6} works in our case with almost no change:  we can define a map $h : B_2^s (r,d,c;l) \to B_1^s(r,d,c;l)$ in the same manner as in \eqref{eq16}, so that a diagram
\[
  \OO_{\PP^2}(-1)^{\oplus c} \overset{\al}{\to} \OO_{\PP^2} ^{\oplus r + c} \overset{\gamma}{\to} Q
\]
is taken to the composition of morphisms in $D(\PP^2)$
\begin{equation*}
\xymatrix{
  & F^\ast[1] \ar[d]^\cong \\
 [\OO_{\PP^2}(-1)^{\oplus c} \ar[r]^\al & \OO_{\PP^2}^{\oplus r+c} \ar[d]^\gamma] \\
  & Q
},
\end{equation*}
where $F$ is the cohomology of the monad given by $\al^\ast$, i.e.\ $F = \kernel (\al^\ast)$.  The argument in the proof of Lemma \ref{lemma6} can then be easily adapted to show that $h$ induces a map $\tilde{h} : B_2^s (r,d,c;l)/\thicksim_{B_2} \to B_1^s(r,d,c;l)/\thicksim_{B_1}$ on the sets of equivalence classes, and that it is bijective.
\end{proof}

\begin{lemma}\label{lemma9}
Suppose $X = \PP^2$ and $c+d=0$.   Then the set $B_2^s(r,d,c; l)/\thicksim_{B_2}$ is in bijection with the set of closed points of a quotient stack.
\end{lemma}
\begin{proof}
Since all 0-dimensional sheaves are Gieseker semistable, for any nonnegative integer $l$, there is a quot scheme $Quot_l$ parametrising the isomorphism classes of  0-dimensional sheaves on $\PP^2$ of length $l$ (see \cite{Niture}, for instance).

Given any element
\[
 \OO (-1)_{\PP^2}^{\oplus c} \overset{\al}{\to} \OO_{\PP^2}^{\oplus r+c} \overset{\gamma}{\to} Q
\]
of the set  $B_2^s(r,d,c;\Lambda)$, we can associate to it a closed point $(\al, Q,\gamma)$ of the scheme
\[
\PP \Hom (\OO (-1)^{\oplus c}_{\PP^2},\OO^{\oplus r+c}_{\PP^2})  \times Quot_l \times  \PP \Hom(\OO_{\PP^2}^{\oplus r+c},Q).
  \]
  Note that the scheme $\PP \Hom(\OO_{\PP^2}^{\oplus r+c},Q)$ only depends on the length $l$ of $Q$.  That is, we have a set-theoretic map
\[
  \phi : B_2^s(r,d,c; l) \to \PP \Hom (\OO (-1)^{\oplus c}_{\PP^2},\OO^{\oplus r+c}_{\PP^2})  \times Quot_l \times \PP \Hom(\OO_{\PP^2}^{\oplus r+c},Q).
\]
The image of $\phi$ is the set of all triples $(\al,Q,\gamma)$ such that $\gamma \al=0$, and $\al$ lies in the open subspace of $\PP \Hom (\OO_{\PP^2} (-1)^{\oplus c}_{\PP^2},\OO^{\oplus r+c}_{\PP^2})$ where  $\cokernel (\al) \in V'$ (here, $V'$ is as defined in the proof of Lemma \ref{lemma7}).  Let $W$ denote the image of $\phi$; then $W$ is a quasi-projective scheme.

We can now adapt the arguments in the proof of Lemma \ref{_quotient_lemma_} to our current situation: the equivalence relation $\thicksim_{B_2}$ on $B_2^s(r,d,c; l)$ induces an equivalence relation on $W$, which we also denote by $\thicksim_{B_2}$.  We can define an action of $G:= \Aut (\OO_{\PP^2}(-1)^{\oplus c}) \times \Aut (\OO_{\PP^2}^{\oplus r+c})$ on $B_2^s(r,d,c; l)$, which induces a well-defined action on $W$ and makes $\phi$ a $G$-equivariant map.  Finally, we show that two elements in $W$ are equivalent with respect to $\thicksim_{B_2}$ if and only if they differ by an action of $G$.  Hence $B_2^s(r,d,c; l)/\thicksim_{B_2}$ is in bijection with the set of closed points of a quotient stack $[W/G]$.
\end{proof}

Let $\mathfrak{M}$ be any moduli stack of  complexes in $D(\PP^2)$ such that any point of $\mathfrak{M}$ corresponds to a complex $E$ such that:
 \begin{itemize}
 \item $H^{-1}(E)$ is torsion-free and $\mu$-semistable;
 \item $H^0(E)$ is a 0-dimensional sheaf;
 \item $H^i(E)=0$ for all $i \neq -1, 0$.
 \end{itemize}
For example, $\mathfrak{M}$ could be a moduli stack of Bridgeland-semistable or polynomial stable complexes (see \cite[Lemma 4.2]{Bayer} or \cite[Lemma 5.3]{LQ}).  Lemma \ref{lemma8} and Lemma \ref{lemma9} together give:

\begin{theorem}\label{P2main-result}
The closed points of the strata of $\mathfrak{M}$ consisting of complexes $E$ on $\PP^2$ further satisfying:
\begin{itemize}
\item $H^{-1}(E)$ is normalised, and
\item $\degree (H^{-1}(E)) - \chi (H^{-1}(E)(-1))=0$,
\end{itemize}
are in 1-1 correspondence with the closed points of some quotient stacks.
\end{theorem}

\section{Proof of Theorem \ref{_theorem_main_}}\label{section-proofmaintheorem}

This section is devoted to the proof of Theorem \ref{_theorem_main_}. We follow closely the idea in the proof of Theorem D in \cite{BR}. In fact,
many of the key lemmas, for instance, Lemma \ref{_lemma_1_}, \ref{_lemma_3_}, and \ref{_lemma_P3_} in this section were adapted from \cite{BR}.
To prove Theorem \ref{_theorem_main_}, we start with the following observations:

\begin{lemma}\label{_lemma_0_}
The Todd classes for $\P^2$ and $\P^3$ are given by
\begin{eqnarray*}
td(\P^2) &=& \left( 1, \frac{3}{2}H, H^2\right);\\
td(\P^3) &=& \left( 1, 2H, \frac{11}{6}H^2, H^3\right),
\end{eqnarray*}
where $H$ is the cohomology class represented by the hyperplane section in $\P^2$ or $\P^3$.
\end{lemma}

\begin{proof}
The proof is standard.
\end{proof}

\begin{lemma}\label{lemma4}
Let $F$ be any sheaf on $\P^3$, and let $ch(F)=(ch_0, ch_1, ch_2, ch_3)$. Let $F_H$ be the restriction of $F$ to
any hyperplane section $\P^2$, then we have $ch(F_H)=(ch_0, ch_1, ch_2)$, i.e. when we restrict a sheaf $F$ from
$\P^3$ to a hyperplane $\P^2$, we can simply drop the last component of the Chern character.
\end{lemma}

\begin{proof}
Let $i: \P^2\to \P^3$ be the inclusion of a hyperplane section.
From the short exact sequence $$0 \to F(-H) \to F \to i_*F_H \to 0,$$ we can obtain
\begin{equation*}
ch(i_*F_H)=\left(0, ch_0, ch_1-\frac12 ch_0, ch_2-\frac12 ch_1+\frac16 ch_0\right).
\end{equation*}
By Grothendieck-Riemann-Roch, we have
$$ch(i_*F_H)\cdot td(\P^3)=i_*(ch(F_H)\cdot td(\P^2)).$$
Applying Lemma \ref{_lemma_0_} and expanding both sides of the equation above, we get
$$ch(F_H)=(ch_0, ch_1, ch_2).$$
\end{proof}

Recall that, for any torsion-free sheaf $F$ of rank $r$ on $\PP^N$ where $N\geq 1$, the restriction $F|_L$ of $F$ to a generic line $L \subset \PP^N$  takes the form $F|_L \cong \oplus_{i=1}^r \OO_L (b_i)$ for some integers $b_1 \geq b_2 \geq \cdots \geq b_r$.  The sequence of integers $\mathbf{b} := (b_1,\cdots,b_r)$ is called the splitting type of $F$; we usually write $\OO_{\PP^N} (\mathbf{b})$ to denote $\oplus_{i=1}^r \OO_{\PP^N} (b_i)$.  Note that $c_1 (\OO_{\PP^N} (\mathbf{b})) = \sum_{i=1}^r b_i$.

\begin{lemma}\label{_lemma_1_}
Let $F$ be a torsion free rank $r$ sheaf on $\P^N$ with splitting type $\b=(b_1, \cdots,  b_r)$,
then we have the following bounds for the dimension of the lowest and highest cohomology groups of $F$:
\begin{eqnarray*}
h^0 F &\leq& h^0\cO_{\P^N}(\b),\\
h^N F &\leq& h^0\cO_{\P^N}(-\b-N-1).
\end{eqnarray*}
\end{lemma}

\begin{proof}
We prove the inequalities by induction on $N$. When $N=1$, the first inequality is an identity, and the second one
simply follows from Serre duality. We let $k$ be an arbitrary integer, and $H$ be a general hyperplane section.

For the first inequality, we have the short exact sequence
\begin{equation*}
0 \to F(k-1) \to F(k) \to F_H(k) \to 0,
\end{equation*}
which implies an exact sequence of cohomology
$$0 \to H^0 F(k-1) \to H^0 F(k) \to H^0 F_H(k).$$ Therefore, we have
$$h^0 F(k) \leq h^0 F(k-1) + h^0 F_H(k).$$ From this and the induction hypothesis, we have
\begin{eqnarray*}
h^0 F(k) - h^0 F(k-1) &\leq& h^0 F_H(k)\\
 &\leq& h^0 \cO_H(\b+k)\\
 &=& h^0\cO_{\P^N}(\b+k) - h^0\cO_{\P^N}(\b+k-1).
\end{eqnarray*}
Note that when $k\ll 0$, we have the vanishing of $h^0 F(k)$ and $h^0 \cO_{\P^N}(\b+k)$. Therefore by adding
the above equation for consecutive values of $k$ we get $$h^0 F\leq h^0 \cO_{\P^N}(\b),$$ proving the first inequality.

For the second inequality, we have an exact sequence of cohomology from the above short exact sequence of sheaves
$$H^{N-1}F_H(k+1) \lra H^N F(k) \lra H^N F(k+1) \lra 0.$$
Therefore, together with the induction hypothesis,
\begin{eqnarray*}
h^N F(k) - h^N F(k+1) &\leq& h^{N-1} F_H(k+1)\\
 &\leq& h^0 \cO_H(-\b-N-k-1)\\
 &=& h^0 \cO_{\P^N}(-\b-N-k-1) - h^0 \cO_{\P^N}(-\b-N-k-2).
\end{eqnarray*}
When $k\gg 0$, we have vanishing of $h^N F(k)$ and $h^0 \cO_{\P^N}(-\b-k-N-1)$. Hence adding both sides for all $k \geq 0$ and cancelling
gives $$h^N F \leq h^0 \cO_{\P^N}(-\b-N-1).$$
\end{proof}

\begin{lemma}\label{_lemma_2_}
For a torsion free sheaf $F$ on $\P^2$ with splitting type $\b=(b_1, \cdots,  b_r)$ and
Chern character $ch(F)=(r, c_1(F), ch_2(F))$, we have
\begin{eqnarray*}
h^0 F &\leq& \displaystyle\sum_{i=1}^r\binom{b_i+2}{2};\\
h^1 F &\leq& -\left(r+\frac32c_1(F)+ch_2(F)\right)+\displaystyle\sum_{i=1}^r\binom{b_i+2}{2};\\
h^2 F &\leq& \displaystyle\sum_{i=1}^r\binom{b_i+2}{2}.
\end{eqnarray*}
\end{lemma}

\begin{proof}
The bound for $h^0$ and $h^2$ follow directly from Lemma \ref{_lemma_1_}.
\begin{eqnarray*}
h^0 F &\leq& h^0\cO_{\P^2}(\b) \leq \displaystyle\sum_{i=1}^r\binom{b_i+2}{2}\\
h^2 F &\leq& h^0\cO_{\P^2}(-\b-3) \leq \displaystyle\sum_{i=1}^r\binom{-b_i-1}{2}=\sum_{i=1}^r\binom{b_i+2}{2}.
\end{eqnarray*}

To get a bound for $h^1 F$, we use $h^1 F = h^0 F + h^2 F -\chi(F)$.  We have
\begin{eqnarray*}
h^0 F+h^2 F &\leq& h^0\cO_{\P^2}(\b)+h^0\cO_{\P^2}(-\b-3)\\
 &\leq& \displaystyle\sum_{i=1}^r(h^0\cO_{\P^2}(b_i) + h^0\cO_{\P^2}(-b_i-3))
\end{eqnarray*}
Note that $h^0\cO_{\P^2}(b_i)$ and $h^0\cO_{\P^2}(-b_i-3)$ cannot be both positive, and
the positive one is equal to $\binom{b_i+2}{2}$, so
$$h^0 F + h^2 F \leq \displaystyle\sum_{i=1}^r\binom{b_i+2}{2}.$$
On the other hand, by Grothendieck-Riemann-Roch and Lemma \ref{_lemma_2_}, we have
\begin{eqnarray*}
\chi(F) &=& \int ch(F)\cdot td(\P^2)\\
&=& \int (r, c_1(F), ch_2(F))\cdot \left(1, \frac32H, H^2\right)\\
&=& r+\frac32 c_1(F)+ ch_2(F).
\end{eqnarray*}
Hence we get the bound for $h^1 F$ as well.
\end{proof}

\begin{coro}\label{_corollary_invariance_}
We can also write the bound for $h^1 F$ in Lemma \ref{_lemma_2_} as
\begin{equation}\label{eq2}
h^1 F\leq -ch_2(F)+ \frac12\sum_{i=1}^r b_i^2.
\end{equation}
 Furthermore, this bound is invariant
under twist and taking dual.
\end{coro}

\begin{proof}
We divide the proof into three steps.

{\it Step 1.} By using the fact that $c_1(F)=\sum_{i=1}^r b_i$, the stated inequality can be translated
directly from the inequality in Lemma \ref{_lemma_2_}.

{\it Step 2.} For the invariance under twist,  we replace $F$ by $F(k)=F\otimes\cO_{\P^2}(kH)$, then
$$b_i(F(k))=b_i+k,$$ and
\begin{eqnarray*}
& &ch(F(k))=ch(F)\cdot ch(\cO_{\P^2}(kH))=ch(F)\cdot \left( 1, kH, \frac{k^2H^2}{2}\right)\\
&=&\left(r, c_1(F)+r\cdot kH, ch_2(F)+c_1(F)\cdot kH+r\cdot \frac{k^2H^2}{2}\right).
\end{eqnarray*}
Now by \eqref{eq2}, we have
\begin{eqnarray*}
h^1 F(k) &\leq& \left(\frac12\sum_{i=1}^r(b_i+kH)^2\right)-ch_2(F(k))\\
&=& \left(\frac12\sum_{i=1}^r b_i^2 + \frac12 r k^2+ k\cdot\sum_{i=1}^r b_i \right) - \left(ch_2(F)+r\cdot \frac12 k^2+\sum_{i=1}^r b_i\cdot k\right)\\
&=& -ch_2(F) + \frac12\sum_{i=1}^r b_i^2
\end{eqnarray*}
This shows the invariance of the upper bound for $h^1F$ under twist.

{\it Step 3.} If we replace $F$ by $F^\ast$, then every $b_i$ becomes $-b_i$, and the Chern character $(ch_0(F), ch_1(F), ch_2(F))$ becomes $(ch_0(F), -ch_1(F), ch_2(F))$. Thus the bound
$\frac12\sum_{i=1}^r b_i^2-ch_2(F)$ remains the same.
\end{proof}

\begin{lemma}\label{_lemma_3_}
For a torsion free sheaf $F$ on $\P^2$ with  splitting type $\b=(b_1, \cdots,  b_r)$ and
Chern character $ch(F)=(r, c_1(F), ch_2(F))$, there exists an integer $Q>0$ which only depends
on $\b$ and $ch(F)$, such that $$h^1 F(k)=h^2 F(k)=h^0 F(-k)=h^1 F(-k)=0$$ for all $k\geq Q$.
\end{lemma}

\begin{proof}
We give a choice for $Q$ in 4 steps. For convenience we denote
\begin{eqnarray*}
b_{max}=\max\{b_1, \cdots, b_r\}\\
b_{min}=\min\{b_1, \cdots, b_r\}.
\end{eqnarray*}

{\it Step 1.} By Lemma \ref{_lemma_1_}, $$h^0 F(-k)\leq h^0\cO_{\P^2}(\b-k).$$
Therefore we have $h^0 F(-k)=0$ when $$k>b_{max}.$$

{\it Step 2.} By Lemma \ref{_lemma_1_}, $$h^2 F(k)\leq h^0\cO_{\P^2}(-\b-k-3).$$
Therefore we have $h^2 F(k)=0$ when $$k>-b_{min}-3.$$

{\it Step 3.} Consider the exact sequence
$$0 \to F(k-1) \to F(k) \to F_H(k) \to 0.$$
Note that here $H=\P^1$ and $F_H(k)=\cO_{\P^1}(\b+k)$. So we have $H^1 F_H(k)=0$
when $k>-b_{min}$ (in fact, $h^1\cO_{\P^1}(m)=0$ whenever $m \geq -1$), in which case  we have the exact sequence of cohomology
$$H^0 F(k) \to H^0 F_H(k) \to H^1 F(k-1) \to H^1 F(k) \to 0.$$
We claim that, if for a certain $k>-b_{min}$,
the map $H^1 F(k-1) \to H^1 F(k)$ is not only surjective but also injective, then for all larger
values of $k$, it is also injective.

Note that, for $k> -b_{min}$, the injectivity of the map $H^1 F(k-1) \to H^1 F(k)$ is equivalent to
the surjectivity of the map $H^0 F(k) \to H^0 F_H(k)$.  Consider the following commutative
diagram
\begin{displaymath}
\xymatrix{
H^0 F(k) \otimes H^0\cO(1) \ar[r] \ar[d] & H^0 F_H(k) \otimes H^0 \cO_H(1) \ar[d]\\
H^0 F(k+1) \ar[r] & H^0 F_H(k+1).
}
\end{displaymath}
The surjectivity of the upper horizontal arrow and the right vertical arrow together imply the surjectivity
of the lower horizontal arrow.  The claim in the last paragraph then follows.  As a result, for $k> -b_{min}$, as $k$ increases, the dimension of $H^1 F(k-1)$ must be strictly decreasing in the beginning, and then becomes constant.  Since for $k$ sufficiently large,
$H^1 F(k)=0$, we have the vanishing of $H^1 F(k)$ for $$k>-b_{min}+h^1 F(-b_{min}),$$
which can be controlled by the formula in Lemma \ref{_lemma_2_}.

{\it Step 4.} By Serre duality, we have
$h^1 F(-k)=h^1 F^\ast (k-3)$.  Therefore, applying the argument in Step 3 to $F^\ast$ (and noting that replacing $F$ by $F^\ast$ changes the sign of each integer in the splitting type), we see that  $H^1 F(-k)$ vanishes for $$k>b_{max}+3+h^1 F^\ast (b_{max}),$$ where the $b_i$ still denote the splitting type of $F$.  The right-hand side of this inequality depends only on the splitting type and the Chern character of $F$ by Lemma
\ref{_lemma_2_}.

Taking the largest one in the bounds we obtained in the above four steps, we have a concrete
estimate of $Q$ for the vanishing of the four cohomology groups in the statement of the Lemma.
\end{proof}

\begin{lemma}\label{_lemma_P3_}
For a torsion free sheaf $F$ on $\P^3$ with  splitting type $\b=(b_1, \cdots,  b_r)$ and
Chern character $ch(F)=(r, c_1(F), ch_2(F), ch_3(F))$, there is an upper bound for the dimension
of its cohomology groups, which only depends on the splitting type and the Chern character of $F$. 
\end{lemma}

\begin{proof}
Similar to in the previous lemma, we give the proof in 3 parts.

{\it Step 1.} By Lemma \ref{_lemma_1_}, we immediately get
\begin{eqnarray*}
h^0 F &\leq& h^0 \cO_{\P^3}(\b);\\
h^3 F &\leq& h^0 \cO_{\P^3}(-\b-4).
\end{eqnarray*}

{\it Step 2.} For an upper bound for $h^2 F$, we use a short exact sequence
$$0 \to F(k) \to F(k+1) \to F_H(k+1) \lra 0,$$
and get an exact sequence for cohomology groups
\begin{eqnarray*}
\cdots &\to & H^1 F_H(k+1) \to \\
\to H^2 F(k) \to H^2 F(k+1) &\to & H^2 F_H(k+1) \to \cdots.
\end{eqnarray*}
Let $Q$ be the constant for the sheaf $F_H$ as  derived in Lemma \ref{_lemma_3_}. When $k\geq Q-1$,
we have $$H^1 F_H(k+1)=H^2 F_H(k+1)=0,$$ which results in $$H^2 F(k)=H^2 F(k+1).$$
By Serre's vanishing theorem, we know $H^2 F(k)$ vanishes because it is the case when $k$ is
sufficiently large. In particular, we have $$H^2 F(Q-1)=0.$$
When $k<Q-1$, the above long exact sequence shows that
$$h^2 F(k)\leq h^2 F(k+1)+h^1 F_H(k+1). $$ We can add all such inequalities for all values of $k$ between $0$
and $Q-2$ and obtain
$$h^2 F\leq h^2 F(Q-1) + \sum_{j=1}^{Q-1} h^1 F_H(j) = \sum_{j=1}^{Q-1} h^1 F_H(j),$$
where all the $h^1 F_H(j)$ have an explicit upper bound by the formula in Lemma \ref{_lemma_2_}.

Note that, here, the choice of $Q$ only depends on the splitting type and the Chern character of $F_H$.  However, the Chern character of $F_H$ depends entirely on that of $F$ by Lemma \ref{lemma4}.  Also, restricting to a generic hyperplane does not alter the splitting type of a torsion-free sheaf.  Hence $Q$ depends only on the splitting type and the Chern character of $F$ itself.

{\it Step 3.} For an upper bound for $h^1 F$, we follow the same procedure as in the previous step.
From the short exact sequence $$0\to F(-k-1) \to F(-k) \to F_H(-k) \to 0,$$ we get the long exact
sequence of cohomology groups
\[
\cdots \to  H^0 F_H(-k) \to H^1 F(-k-1) \to H^1 F(-k) \to H^1 F_H(-k) \to \cdots.
\]
When $k\geq Q$, we have $$H^0 F_H(-k)=H^1 F_H(-k)=0$$ by Lemma \ref{_lemma_3_}. Therefore
$$H^1 F(-k)=H^1 F(-k-1),$$ which vanishes due to Serre's vanishing theorem. In particular, $$H^1 F(-Q)=0.$$
When $k<Q$, we have $$h^1 F(-k)\leq h^1 F(-k-1) + h^1 F_H(-k).$$ Summation and cancellation
results in $$h^1 F\leq h^1 F(-Q)+\sum_{j=0}^{Q-1} h^1 F_H(-j)=\sum_{j=0}^{Q-1}h^1 F_H(-j),$$
where $h^1 F_H(-j)$ can still be bounded by the formula in Lemma \ref{_lemma_2_}.
\end{proof}

So far we have derived all necessary results for general torsion free sheaves on $\P^2$ and $\P^3$,
from now on we want to apply these results to reflexive semistable sheaves.

\begin{lemma}
For a reflexive $\mu$-semistable sheaf $F$ of splitting type $\b=(b_1, b_2, \cdots, b_r)$ on  $\P^3$,
assuming that $b_1\geq \cdots \geq b_r$, we have $$b_i-b_{i+1}\leq 2$$ for each $1\leq i\leq r-1$.
\end{lemma}

\begin{proof}
We consider the grassmannian $G(1,3)$ of all lines in $\P^3$, which is a quadric hypersurface in $\P^5$
via the Pl\"{u}cker embedding. By the definition of splitting type, there is an open sense subset $U\subset G(1,3)$,
such that for any line $L\in U$, the restriction of $F$ on $L$ decomposes in the same way as specified by
the splitting type.

Choose a smooth hyperplane section $S\subset G(1,3)$ whose intersection with $U$ is non-empty. Then
$S\cap U$ is open and dense in $S$. For any line $L\in S \cap U$, the decomposition of the restriction of
$F$ on $L$ agrees with the splitting type. By \cite[Proposition 7.1]{EHV}, we know
that any gap in the splitting type is at most 2.
\end{proof}

\begin{coro}\label{lemma5} 
For a reflexive $\mu$-semistable sheaf of rank $r$ on $\mathbb{P}^3$ with splitting type $\mathbf{b} = (b_1,\cdots,b_r)$, we have
\begin{equation}\label{eq3}
\pm b_i \leq \frac{|c_1|}{r}+r.
\end{equation}
for each $i$.
\end{coro}

\begin{lemma}\label{_lemma_vanishing_constant_}
Let $E$ be a $\mu$-semistable reflexive sheaf on $\P^3$ with splitting type $\mathbf{b} = (b_1,\cdots,b_r)$,
and $F$ be the restriction of $E$ to a generic hyperplane. Then the constant $Q$ for $F$ in Lemma \ref{_lemma_3_} can be chosen as
$$Q=\frac{|c_1|}{r}+r+4-ch_2 + \frac12\sum_{i=1}^r b_i^2.$$
\end{lemma}

\begin{proof}
From the proof of Lemma \ref{_lemma_3_}, we know that the vanishing conditions for the four cohomology groups $H^1F(k), H^2F(k), H^0F(-k),H^1F(-k)$ are:
 \begin{align*}
   k &>b_{max},\\
    k &>-b_{min}-3, \\
    k &>-b_{min}+h^1 F(-b_{min}), \text{ and } \\
   k &>b_{max}+3+h^1 F^\ast (b_{max}).
 \end{align*}
By Corollary \ref{lemma5}, we have that
\begin{equation*}
\pm b_{max}, \pm b_{min}\leq \frac{|c_1|}{r}+r.
\end{equation*}
And by Lemma \ref{_corollary_invariance_}, we have that
\begin{eqnarray*}
h^1 F(-b_{min}) \leq -ch_2 + \frac12\sum_{i=1}^r b_i^2;\\
h^1 F^\ast (b_{max}) \leq -ch_2 + \frac12\sum_{i=1}^r b_i^2.
\end{eqnarray*}
Hence  the value of $Q$ as proposed in the statement of the lemma.
\end{proof}

Now we summarize all the above results for sheaves on $\P^2$ to get a bound for the Euler characteristic of
a $\mu$-semistable reflexive sheaf on $\P^3$.

\begin{lemma}\label{_lemma_euler_}
For any $\mu$-semistable reflexive sheaf $F$ on $\P^3$ with  splitting type
$\b=(b_1, \cdots,  b_n)$ and Chern character $ch(F)=(n, c_1, ch_2, ch_3)$, there is a
bound for its Euler characteristic $\chi(F)$, depending only on $ch_0(F),ch_1(F)$ and $ch_2(F)$.
\end{lemma}

\begin{proof}
We first note that, by Lemma \ref{lemma4}, for any hyperplane $H$ in $\P^3$, we have $$ch(F_H)=(n, c_1, ch_2).$$
By Lemma \ref{_lemma_P3_} and Corollary \ref{_corollary_invariance_}, we obtain
\begin{eqnarray*}
h^2 F &\leq& \sum_{i=1}^{Q-1} h^1 F_H(i) \leq Q\cdot \left(-ch_2 + \frac12\sum_{i=1}^r b_i^2\right);\\
h^1 F &\leq& \sum_{i=0}^{Q-1} h^1 F_H(-i) \leq Q\cdot \left(-ch_2 + \frac12\sum_{i=1}^r b_i^2\right).
\end{eqnarray*}
Hence
\begin{eqnarray*}
& & h^1 F + h^2 F <2Q \cdot \left(-ch_2 + \frac12\sum_{i=1}^r b_i^2\right)\\
&=& 2\left(\frac{|c_1|}{r}+r+4-ch_2+\frac12\sum_{i=1}^r b_i^2\right)\left(-ch_2+\frac12\sum_{i=1}^r b_i^2\right) \text{ (by Lemma \ref{_lemma_vanishing_constant_})}\\
&\leq& 2\left(\frac{|c_1|}{r}+r+4-ch_2+\frac12\sum_{i=1}^r \left(\frac{|c_1|}{r}+r\right)^2\right)\left(-ch_2+\frac12\sum_{i=1}^r \left(\frac{|c_1|}{r}+r\right)^2\right),
\end{eqnarray*}
where the last inequality follows from Corollary \ref{lemma5}.

Furthermore, by Lemma \ref{_lemma_1_}, we have
\begin{eqnarray*}
h^0 F &\leq& h^0 \cO_{\P^3}(\b) = \sum_{i=1}^r h^0\cO_{\P^3}(b_i);\\
h^3 F &\leq& h^0 \cO_{\P^3}(-\b-4) = \sum_{i=1}^r h^0\cO_{\P^3}(-b_i-4).
\end{eqnarray*}
Note that for every $b_i$, $h^0\cO_{\P^3}(b_i)$ and $h^0\cO_{\P^3}(-b_i-4)$ cannot be both positive,
therefore,
$$h^0\cO_{\P^3}(b_i) + h^0\cO_{\P^3}(-b_i-4) \leq \left| \frac{(b_i+3)(b_i+2)(b_i+1)}{3\cdot 2\cdot 1} \right|.$$
By Corollary \ref{lemma5} again, for each $i$ we have $b_i\leq\frac{|c_1|}{r}+r$, and hence
$$h^0\cO_{\P^3}(b_i) + h^0\cO_{\P^3}(-b_i-4) < \frac16 \left( \frac{|c_1|}{r}+r+3 \right)^3.$$
So finally we have $$h^0 F+h^3 F<\frac{r}{6} \left( \frac{|c_1|}{r}+r+3 \right)^3.$$

Combining the above results, we get a  bound for $\chi (F)$, the Euler characteristic of $F$, in terms of only $ch_0=n, ch_1=c_1$ and $ch_2$ and not involving $ch_3$:
\begin{eqnarray*}
|\chi(F)| &\leq& h^0 F+h^1 F+h^2 F+h^3 F\\
&<& 2\left(\frac{|c_1|}{r}+r+4-ch_2+\frac12\sum_{i=1}^r \left(\frac{|c_1|}{r}+r\right)^2 \right)\\
& & \cdot\left( -ch_2+\frac12\sum_{j=1}^r \left(\frac{|c_1|}{r}+r\right)^2\right)+\frac{r}{6} \left( \frac{|c_1|}{r}+r+3 \right)^3.
\end{eqnarray*}
\end{proof}

Finally, we are ready to prove Theorem \ref{_theorem_main_}, the Bogomolov-type inequality for $\mu$-semistable reflexive sheaves on $\mathbb{P}^3$.

\begin{proof}[Proof of Theorem \ref{_theorem_main_}]
Continuing the notation above, we write $r=ch_0(F)$ and $c_1=ch_1(F)$ in this proof.
By Grothendieck-Riemann-Roch theorem and Lemma \ref{_lemma_0_},
\begin{eqnarray*}
\chi(F) &=& \int ch(F)\cdot td(\P^3)\\
          &=& \int (r, c_1, ch_2, ch_3)\cdot \left(1, 2H, \frac{11}{6}H^2, H^3\right)\\
          &=& ch_3+2ch_2+\frac{11}{6}c_1+r
\end{eqnarray*}
Then by Lemma \ref{_lemma_euler_}, we obtain \eqref{eq18}.
\end{proof}

\appendix
\section{The two stacks in theorem \ref{_quotient_coro_} are different}\label{app}

In this section, we make an observation that, the moduli stack $\Ac_1^s(ch;l)$ and the quotient stack $[W/G]$ that we considered in Theorem \ref{_quotient_coro_} are \emph{not} isomorphic, although their underlying sets of closed points are bijective.

We consider a very special case: $l=0$. Then any object in $A_1^s(ch;0)$ is a complex $E$ with the only non-trivial cohomology $F=H^{-1}(E)$, which is itself a stable reflexive sheaf. Therefore the definition simply reduces to
\[
A_1^s(ch;0)=\{\text{stable reflexive sheaf } F: ch(F)=ch\}.
\]
And the corresponding stratum $\Ac_1^s(ch;0)$ of the moduli stack \eqref{eq1} becomes the moduli stack of these stable reflexive sheaves with prescribed Chern character $ch$. For $A_3^s(ch;0)$, due to the lack of cohomology in degree 0 for any complex in $A_1^s(ch;0)$, the quotient sheaf $Q$ becomes trivial. It's obvious ``taking dual" is no longer necessary and the definition of $A_3^s(ch;0)$ simply reduces to
\begin{align*}
A_3^s(ch;0)=\{\text{morphism }g: R^{-1} \rightarrow R^0: g \text{ is injective, and } \cokernel (g) \\
\text{ is any stable reflexive sheaf in } A_1^s(ch;0)\}.
\end{align*}

We take a close look at of the two stacks. For the ``moduli" stack $\Ac_1^s(ch;0)$, since it parametrizes simple objects, by \cite[Corollary 4.3]{Lie}, its inertia stack is naturally identified as $\mathbf{G}_m$. In other words, the stabilizer at every closed point of this stack is always $\mathbf{G}_m$. In fact, if we ignore this trivial $\mathbf{G}_m$ action, it is proved in \cite[Theorem 2.2]{MR} that these sheaves can be parametrized by an irreducible non-singular scheme.

On the other hand, due to the vanishing of $Q$, the ``quotient" stack structure becomes $[V'/G]$. We claim that the inertia stack of this quotient stack is not $\mathbf{G}_m$. In fact, take any $g: R^{-1} \rightarrow R^0$ in $V'$, the stabilizer $I_g$ of the $G$ action at this point is given by
\[
I_g=\{(r^{-1},r^0): r^{-1}\in \Aut (R^{-1}), r^0 \in \Aut (R^0), gr^{-1}=r^0g\}.
\]
It has been computed in \cite[Proof of Proposition 2.9]{MR}, that under our assumptions in \eqref{eq4}, we have
\[
\dim I_g=1+\binom{c_2-2r+1}{3}.
\]
In fact, it's easy to find that under the assumptions in \eqref{eq4}, we have
$$c_2-2r+1 \geqslant 4,$$
therefore $\dim I_g$ always has a much larger dimension than $\mathbf{G}_m$. Together with our knowledge of the moduli stack discussed above, we conclude that the two stacks $\Ac_1^s(ch;0)$ is \emph{not} the same as the quotient stack structure in Lemma \ref{_quotient_lemma_}, and therefore Theorem \ref{_quotient_coro_}.

The situation for arbitrary values of $l$ comes with the same spirit, namely, the group $G$ in the quotient stack is too large and produces extra ``stacky" structure which doesn't appear in the moduli stack. For our purpose, the above special case has already showed us that, a bijection between the closed points of the two stacks is the best we can hope for. However, the result still provides us a very concrete way to understand some Bridgeland stable objects in the derived category, which a priori cannot be not so explicitly described.

\end{document}